%% file: main_arXiv.tex
\documentclass[10pt]{article}

\usepackage{amsmath,amssymb,amsthm}
\usepackage{mathtools,nccmath,mathrsfs}
\usepackage[linesnumbered,ruled]{algorithm2e}

\usepackage{enumitem}

\usepackage{graphicx}
\usepackage{subcaption}
\usepackage{xcolor}
\usepackage{url}

\usepackage{comment}

\usepackage{geometry}
\usepackage[noblocks]{authblk}

\DeclareMathOperator*{\argmin}{arg\,min}

\newtheorem{theorem}{Theorem}
\newtheorem{lemma}{Lemma}
\newtheorem{proposition}{Proposition}

\theoremstyle{definition}
\newtheorem{definition}{Definition}
\newtheorem{assumption}{Assumption}

\theoremstyle{remark}
\newtheorem{remark}{Remark}

\allowdisplaybreaks

\title{\bf
Communication-Efficient Distributed SGD with Compressed Sensing
}

\author[1]{Yujie Tang}
\author[1]{Vikram Ramanathan}
\author[2]{Junshan Zhang}
\author[1]{Na Li}

\affil[1]{\small School of Engineering and Applied Sciences,
Harvard University}
\affil[2]{\small School of Electrical, Computer and Energy Engineering, Arizona State University}

\date{}

\begin{document}

\maketitle

%%%%%%%%%%%%%%%%%%%%%%%%%%%%%%%%%%%%%%%%%%%%%%%%%%%%%%%%%%%%%%%%%%%%%%%%%%%%%%%%
\begin{abstract}

We consider large scale distributed optimization over a set of edge devices connected to a central server, where the limited communication bandwidth between the server and edge devices imposes a significant bottleneck for the optimization procedure. Inspired by recent advances in federated learning, we propose a distributed stochastic gradient descent (SGD) type algorithm that exploits the sparsity of the gradient, when possible, to reduce communication burden. At the heart of the algorithm is to use compressed sensing techniques for the compression of the local stochastic gradients at the device side;  and at the server side, a sparse approximation of the global stochastic gradient is recovered from the noisy aggregated compressed local gradients. We conduct theoretical analysis on the convergence of our algorithm in the presence of noise perturbation incurred by the communication channels, and also conduct numerical experiments to corroborate its effectiveness.

\end{abstract}

%%%%%%%%%%%%%%%%%%%%%%%%%%%%%%%%%%%%%%%%%%%%%%%%%%%%%%%%%%%%%%%%%%%%%%%%%%%%%%%%
\section{Introduction}\label{sec:intro}
\input{introduction_arXiv.tex}

\section{Problem Setup}
Consider a group of $n$ edge devices and a server. Each device $i$ is associated with a differentiable local objective function $f_i:\mathbb{R}^d\rightarrow\mathbb{R}$, and is able to query a stochastic gradient $\mathsf{g}_i(x)$ such that $\mathbb{E}[\mathsf{g}_i(x)]=\nabla f_i(x)$.
Between each device and the server are an uploading communication link and a broadcasting communication link.
The goal is to solve
\begin{equation}\label{eq:main_problem}
\min_{x\in\mathbb{R}^d}\ \ 
f(x)\coloneqq \frac{1}{n}\sum\nolimits_{i=1}^n f_i(x)
\end{equation}
through queries of stochastic gradients at each device and exchange of information between the server and each device.

One common approach for our problem setup is the stochastic gradient descent (SGD) method: For each time step $t$, the server first broadcasts the current iterate $x(t)$ to all devices, and then each device produces a stochastic gradient $g_i(t)=\mathsf{g}_i(x(t))$ and uploads it to the server, after which the server updates $x(t\!+\!1)=x(t)-\eta\cdot\frac{1}{n}\sum_i g_i(t)$. However, as the server needs to collect local stochastic gradients from each device at every iteration, the vanilla SGD may encounter significant bottleneck imposed by the uploading links if $d$ is very large. This issue may be further exacerbated if the server and the devices are connected via lossy wireless networks of limited bandwidth, which is the case for many IoT applications.

In this work, we investigate the situation where the communication links, particularly the uploading links from each edge device to the server, have limited bandwidth that can significantly slow down the whole optimization procedure; the data transmitted through each uploading link may also be corrupted by noise. Our goal is to develop an SGD-type algorithm for solving~\eqref{eq:main_problem} that achieves better communication efficiency over the uploading links.

\section{Algorithm}

Our algorithm is outlined in Algorithm~\ref{alg:main},
which is based on the SGD method with the following major ingredients:
\begin{enumerate}[leftmargin=0pt,topsep=2pt,itemindent=15pt,labelwidth=6pt,labelsep=7pt,listparindent=15pt,itemsep=2pt]
\item Compression of local stochastic gradients using compressed sensing techniques. Here each edge device compresses its local gradient by $y_i(t)=\Phi g_i(t)$ before uploading it to the server. The matrix $\Phi\in\mathbb{R}^{Q\times d}$ is called the \emph{sensing matrix}, and its number of rows $Q$ is strictly less than the number of columns $d$. As a result, the communication burden of uploading the local gradient information can be reduced.

We emphasize that Algorithm~\ref{alg:main} employs the \emph{for-all} scheme of compressed sensing, which allows one $\Phi$ to be used for the compression of all local stochastic gradients (see Section~\ref{subsec:compressed_sensing} for more details on the \emph{for-each} and the \emph{for-all} schemes).

After collecting the compressed local gradients and obtaining $\frac{1}{n}\sum_{i=1}^n y_i(t)$ (corrupted by communication channel noise), the server recovers a vector $\Delta(t)$ by a compressed sensing algorithm, which will be used for updating $x(t)$.

\item Error feedback of compressed gradients. In general, the compressed sensing reconstruction %only recovers an approximation of the original vector, which
will introduce a nonzero bias in the SGD iterations that hinders convergence. To handle this bias, we adopt the error feedback method in~\cite{stich2018sparsified,karimireddy2019error} and modify it similarly as FetchSGD~\cite{rothchild2020fetchsgd}. The resulting error feedback procedure is done purely at the server side without knowing the true aggregated stochastic gradients.
\end{enumerate}

Note that the aggregated vector $\tilde{y}(t)$ is corrupted by additive noise $w(t)$ from the uploading links. This noise model incorporates a variety of communication schemes, including digital transmission with quantization, and over-the-air transmission for wireless multi-access networks~\cite{junshan2021federated}.

\begin{algorithm}[t]
\DontPrintSemicolon
\textbf{Input}: sparsity level $K$, size of sensing matrix $Q \!\times\! d$, step size $\eta$, number of iterations $T$, initial point $x_0$ \;
Initialize: $x(1)=x_0$, $\varepsilon(1)=0$\;
The server generates the sensing matrix $\Phi\in\mathbb{R}^{Q\times d}$ and sends it to every edge device\;
\For{$t=1,2\ldots,T$}{
The server sends $x(t)$ to every edge device\;
\ForEach{device $i=1,\ldots,n$}{
Device $i$ samples a stochastic gradient $g_i(t)=\mathsf{g}_i(x(t))$\;
Device $i$ constructs $y_i(t)=\Phi g_i(t)\in\mathbb{R}^Q$\;
Device $i$ sends $y_i(t)$ back to the server\;
}
The server receives $\tilde{y}(t)=\frac{1}{n}\sum_{i=1}^n y_i(t)+w(t)$, where $w(t)$ denotes additive noise incurred by the communication channels\;
The server computes $z(t)=\eta \tilde{y}(t)+\varepsilon(t)$\;
The server reconstructs $\Delta(t)=\mathcal{A}(z(t);\Phi)$, where $\mathcal{A}(z(t);\Phi)$ denotes the output of the compressed sensing algorithm of choice

The server updates $x(t\!+\!1)=x(t)-\Delta(t)$\;
The server updates $\varepsilon(t\!+\!1)=z(t)-\Phi\Delta(t)$\;
}
\caption{SGD with Compressed Sensing}
\label{alg:main}
\end{algorithm}

We now provide more details on our algorithm design.

\subsection{Preliminaries on Compressed Sensing}\label{subsec:compressed_sensing}

Compressed sensing \cite{candes2008introduction} is a technique that allows efficient sensing and reconstruction of an approximately sparse signal. Mathematically, in the sensing step, a signal $x\in\mathbb{R}^d$ is observed through linear measurement $y=\Phi x+w$, where $\Phi\in\mathbb{R}^{Q\times d}$ is a pre-specified sensing matrix with $Q<d$, and $w\in\mathbb{R}^Q$ is additive noise. Then in the reconstruction step, one recovers the original signal $x$ by approximately solving
\begin{align}
\hat{x}=\,&\argmin\nolimits_{z}\, \|z\|_0
\quad\textrm{s.t.}\ \ y=\Phi z,
&
\textrm{($w=0$)}
\label{eq:comp_sens_form1}
\\
\hat{x}=\,&\argmin\nolimits_{z}\, \mfrac{1}{2}\|y \!-\! \Phi z\|_2^2\quad
\textrm{s.t.}\ \ \|z\|_0\leq K,
&
\textrm{($w\neq 0$)}
\label{eq:comp_sens_form2}
\end{align}
where $K$ restricts the number of nonzero entries in $\hat{x}$.

Both~\eqref{eq:comp_sens_form1} and~\eqref{eq:comp_sens_form2} are NP-hard nonconvex problems\cite{natarajan1995sparse,davis1997adaptive}
, and researchers have proposed various compressed sensing algorithms for obtaining approximate solutions. As discussed below, the reconstruction error $\|\hat{x}-x\|$ will heavily depend on i) the design of the sensing matrix $\Phi$, and ii) whether the signal $x$ can be well approximated by a sparse vector.

\vspace{8pt}
\noindent\textbf{Design of the sensing matrix $\Phi$.}
Compressed sensing algorithms can be categorized into two schemes~\cite{gilbert2010sparse}: i)~the \emph{for-each} scheme, in which a probability distribution over sensing matrices is designed to provide desired reconstruction for a fixed signal, and every time a new signal is to be measured and reconstructed, one needs to randomly generate a new $\Phi$; ii)~the \emph{for-all} scheme, in which a single $\Phi$ is used for the sensing and reconstruction of all possible signals.\footnote{A more detailed explanation of the two schemes can be found in Appendix~\ref{appendix:two_schemes}.}
We mention that count sketch is an example of a for-each scheme algorithm.
In this paper, we choose the for-all scheme so that the server doesn't need to send a new matrix to each device per iteration.

To ensure that the linear measurement $y=\Phi x$ can discriminate approximately sparse signals, researchers have proposed the \emph{restricted isometry property} (RIP) \cite{candes2008introduction} as a condition on~$\Phi$:

\begin{definition}\label{def:RIP}
We say that $\Phi\in\mathbb{R}^{Q\times d}$ satisfies the $(K,\delta)$-restricted isometry property, if $(1\!-\!\delta)\|x\|_2^2\leq \|\Phi x\|_2^2\leq (1\!+\!\delta)\|x\|_2^2$ for any $x\in\mathbb{R}^d$ that has at most $K$ nonzero entries.
\end{definition}

The restricted isometry property on $\Phi$ is fundamental for analyzing the reconstruction error of many compressed sensing algorithms under the for-all scheme~\cite{foucart2012sparse}.

\vspace{8pt}
\noindent\textbf{Metric of sparsity.} The classical metric of sparsity is the $\ell_0$ norm defined as the number of nonzero entries. However, for our setup, the vectors to be compressed can only be approximately sparse in general, which cannot be handled by the $\ell_0$ norm as it is not stable under small perturbations. Here, we adopt the following sparsity metric from~\cite{lopes2016unknown}:\footnote{Compared to~\cite{lopes2016unknown}, we add a scaling factor $d^{-1}$ so that $\operatorname{sp}(x)\!\in\!(0,1]$.}
\begin{equation}
\operatorname{sp}(x)
\coloneqq \|x\|_1^2/(\|x\|_2^2\cdot d),
\qquad x\in\mathbb{R}^d\backslash\{0\}.
\end{equation}
The continuity of $\operatorname{sp}(x)$ indicates that $\operatorname{sp}(x)$ is robust to small perturbations on $x$, and it can be shown that $\operatorname{sp}(x)$ is \emph{Schur-concave}, meaning that it can characterize \emph{approximate sparsity} of a signal. $\operatorname{sp}(x)$ has also been used in
\cite{lopes2016unknown} for performance analysis of compressed sensing algorithms. 

\subsection{Details of Algorithm Design}\label{subsec:algorithm_design}

\noindent\textbf{Generation of $\Phi$. } As mentioned before, we choose compressed sensing under the \emph{for-all} scheme for gradient compression and reconstruction.
We require that the sensing matrix $\Phi\in\mathbb{R}^{Q\times d}$ have a low storage cost, since it will be transmitted to and stored at each device; $\Phi$ should also satisfy RIP so that the compressed sensing algorithm $\mathcal{A}$ has good reconstruction performance. The following proposition suggests a storage-friendly approach for generating matrices satisfying RIP.

\begin{proposition}[\cite{haviv2017restricted}]\label{prop:subsampled_Fourier}
Let $B\in\mathbb{R}^{d\times d}$ be an orthogonal matrix with entries of absolute values $O(1/\sqrt{d})$, and let $\delta>0$ be sufficiently small. For some $Q = \tilde{O}(\delta^{-2}K\log^2
K\log d)$,\footnote{
The $\tilde{O}$ notation hides logarithm dependence on $1/\delta$.
}
let $\Phi\in\mathbb{R}^{Q\times d}$ be a matrix whose $Q$ rows are chosen uniformly and independently from the rows of $B$, multiplied by $\sqrt{d/Q}$. Then, with high
probability, $\Phi$ satisfies the $(K,\delta)$-RIP.
\end{proposition}

This proposition indicates that, we can choose a ``base matrix'' $B$ satisfying the condition in Proposition~\ref{prop:subsampled_Fourier}, and then randomly choose $Q$ rows to form $\Phi$. In this way, $\Phi$ can be stored or transmitted by merely the corresponding row indices in $B$. Note that Proposition~\ref{prop:subsampled_Fourier} only requires $Q$ to have logarithm dependence on $d$. Candidates of the base matrix $B$ include the discrete cosine transform (DCT) matrix and the Walsh-Hadamard transform (WHT) matrix, as both DCT and WHT and their inverse have fast algorithms of time complexity $O(d\log d)$, implying that multiplication of $\Phi$ or $\Phi^{\!\top}$ with any vector can be finished within $O(d\log d)$ time.

\vspace{8pt}
\noindent\textbf{Choice of the compressed sensing algorithm.} We let $\mathcal{A}$ be the \emph{Fast Iterative Hard Thresholding} (FIHT) algorithm~\cite{wei2014fast}.\footnote{See also Appendix~\ref{appendix:FIHT} for a brief summary.}
Our experiments suggest that FIHT achieves a good balance between computation efficiency and empirical reconstruction error compared to other algorithms we have tried.

We note that FIHT has a tunable parameter $K$ that controls the number of nonzero entries of $\Delta(t)$. This parameter should accord with the sparsity of the vector to be recovered (see Section~\ref{subsec:theory} for theoretical results).
In addition, the server can broadcast the sparse vector $\Delta(t)$ instead of the whole $x(t)$ for the edge devices to update their local copies of $x(t)$, which saves communication over the broadcasting links.

\vspace{8pt}
\noindent\textbf{Error feedback.}
We adopt error feedback to facilitate convergence of Algorithm~\ref{alg:main}. The following lemma verifies that Algorithm~\ref{alg:main} indeed incorporates the error feedback steps in~\cite{karimireddy2019error}; the proof is straightforward which we omit here.
\vspace{-2pt}
\begin{lemma}\label{lemma:error_feedback}
Consider Algorithm~\ref{alg:main}, and suppose $\Phi$ is generated according to Proposition~\ref{prop:subsampled_Fourier}. Then for each $t$, there exist unique $p(t)\in\mathbb{R}^d$ and $e(t)\in\mathbb{R}^d$ satisfying $z(t)=\Phi p(t) + \eta w(t)$, $\varepsilon(t)=\Phi e(t)$, $e(1)=0$ such that
%\vspace{-5pt}
\begin{equation}
\begin{aligned}
p(t) =\ & \eta g(t) + e(t),
\quad
x(t\!+\!1) = x(t)-\Delta(t), \\
\Delta(t) =\ & \mathcal{A}\left(
\Phi p(t) + \eta w(t);\Phi
\right), \\
e(t\!+\!1) =\ & p(t)-\Delta(t)
+\mfrac{\eta Q}{d}\Phi^\top w(t).
\end{aligned}
\end{equation}
where $g(t)\coloneqq \frac{1}{n}\sum\nolimits_{i=1}^n g_i(t)$.
\end{lemma}
\vspace{-2pt}

By comparing Lemma~\ref{lemma:error_feedback} with \cite[Algorithm 2]{karimireddy2019error}, we see that the only difference lies in the presence of communication channel noise $w(t)$ in our setting.
In addition, since error feedback is implemented purely at the server side, the edge devices will be \emph{stateless} during the whole optimization procedure.

\subsection{Theoretical Analysis}\label{subsec:theory}

% We now provide convergence guarantees of Algorithm~\ref{alg:main}.

First, we make the following assumptions on the objective function $f$, the stochastic gradients $g_i(t)$ and the communication channel noise $w(t)$:
\begin{assumption}
The function $f(x)$ is $L$-smooth over $x\in\mathbb{R}^d$, i.e., there exists $L>0$ such that for all $x,y\in\mathbb{R}^n$,
$
\|\nabla f(x)-\nabla f(y)\|_2\leq L \|x-y\|_2.
$
\end{assumption}

\begin{assumption}\label{assumption:var_grad}
Denote $g(t)=\frac{1}{n}\sum_{i=1}^n g_i(t)$. There exists $G>0$ such that
$
\mathbb{E}\big[
\left\|
g(t)-\nabla f(x(t))\right\|_2^2\big]
\leq G^2
$
for all $t$.
\end{assumption}

\begin{assumption}\label{assumption:second_moment_w}
The communication channel noise $w(t)$ satisfies $\mathbb{E}[\|w(t)\|_2^2]\leq \sigma^2$ for each $t$.
\end{assumption}

% Note that Assumption~\ref{assumption:var_grad} bounds the variance rather than the second moment of $g(t)$, which allows $\|\nabla f(x)\|_2$ to be unbounded over $x\in\mathbb{R}^d$. Assumption~\ref{assumption:second_moment_w} only bounds the second moment of $w(t)$ and imposes no further assumption on its mean.

Our theoretical analysis will be based on the following result on the reconstruction error of FIHT:

\begin{lemma}[{\cite[Corollary I.3]{wei2014fast}}]\label{lemma:FIHT_performance}
Let $K$ be the maximum number of nonzero entries of the output of FIHT. Suppose the sensing matrix $\Phi\in\mathbb{R}^{Q\times d}$ satisfies $(4K,\delta_{4K})$-RIP for sufficiently small $\delta_{4K}$. Then, for any $x\in\mathbb{R}^d$ and $w\in\mathbb{R}^Q$,
\begin{equation}\label{eq:FIHT_performance}
\begin{aligned}
\|\mathcal{A}(\Phi x + w;\Phi)
-x\|_2
\leq 
(C_{\mathcal{A},\mathrm{s}} \!+\! 1)\big\|x \!-\! x^{[K]}\big\|_2
+\mfrac{C_{\mathcal{A},\mathrm{s}}}{\sqrt{K}}\big\|x \!-\! x^{[K]}\big\|_1
+ C_{\mathcal{A},\mathrm{n}} \|w\|_2
\end{aligned}
\end{equation}
where $C_{\mathcal{A},\mathrm{s}}$ and $C_{\mathcal{A},\mathrm{n}}$ are positive constants that depend on $\delta_{4K}$.
\end{lemma}

We are now ready to establish convergence of Algorithm~\ref{alg:main}.
\begin{theorem}\label{theorem:main}
Let $K$ be the maximum number of nonzero entries of the output of FIHT. Suppose the sensing matrix $\Phi\in\mathbb{R}^{Q\times d}$ satisfies $(4K,\delta_{4K})$-RIP for sufficiently small $\delta_{4K}$. Furthermore, assume that
\begin{equation}\label{eq:sparsity_condition}
\operatorname{sp}(p(t))
\leq \gamma\cdot \frac{2 K/d}{[1+C_{\mathcal{A},\mathrm{s}}(3-2K/d)]^2}
\end{equation}
for all $t\geq 1$ for some $\gamma\in(0,1)$, where $p(t)$ is defined in Lemma~\ref{lemma:error_feedback}. Then for sufficiently large $T$, by choosing $\eta=1/(L\sqrt{T})$, we have that
\begin{align*}
\frac{1}{T}\sum_{t=1}^T
\mathbb{E}\!\left[\|\nabla\!f(x(t)\mkern-1mu)\|_2^2\right]
\leq
\frac{6L\!\left(f(x(1)) \!-\! f^\ast\right) + 3G^2}{2\sqrt{T}}
+\frac{3}{T}\!\left[
\frac{\gamma(1\!+\!\gamma)}{(1\!-\!\gamma)^2}G^2
+\frac{ 2\big(C_{\mathcal{A},\mathrm{n}}
\!+\!\sqrt{Q/d}\big)^{\!2}\sigma^2}{1-\gamma}
\right]\!.
\end{align*}
In addition, if $f$ is convex and has a minimizer $x^\ast\in\mathbb{R}^d$, then we further have
\begin{align*}
\mathbb{E}\!\left[f(\bar{x}(t)) \!-\! f^\ast\right]
\leq
\frac{L\|x(1)-x^\ast\|_2^2+G^2/L}{\sqrt{T}}
+
\frac{6}{T}\!\left[
\frac{\gamma(1\!+\!\gamma)G^2}{(1\!-\!\gamma)^2L}
+ \frac{2\big(C_{\mathcal{A},\mathrm{n}}
\!+\!\sqrt{Q/d}\big)^{\!2}\sigma^2}{(1\!-\!\gamma)L}
\right]\!,
\end{align*}
where $\bar{x}(t)\coloneqq \frac{1}{T}\sum_{t=1}^T x(t)$ and $f^\ast\coloneqq f(x^\ast)$.
\end{theorem}

Proof of Theorem~\ref{theorem:main} is given in Appendix~\ref{appendix:proof}.

\begin{comment}
\begin{remark}
Theorems~\ref{theorem:main_nonconvex} and~\ref{theorem:main_convex} do not require each $f_i$ to have bounded discrepancy, which is different from the theoretical guarantees of federated-averaging-type algorithms~\cite{kairouz2019advances}. This suggests that our algorithm may be favored for handling federated learning with each worker having non-i.i.d. data.
\end{remark}
\end{comment}

\begin{remark}\label{remark:sparsity_p}
Theorem~\ref{theorem:main} requires $\operatorname{sp}(p(t))$ to remain sufficiently low. This condition is hard to check and can be violated in practice (see Section~\ref{sec:simulation}). However, our numerical experiments seem to suggest that even if the condition~\eqref{eq:sparsity_condition} is violated, Algorithm~\ref{alg:main} may still exhibit relatively good convergence behavior when the gradient $g(t)$ itself has relatively low sparsity level. Theoretical investigation on these observations will be interesting future directions.
\end{remark}

\section{Numerical Results}\label{sec:simulation}

\subsection{Test Case with Synthetic Data}

We first conduct numerical experiments on a test case with synthetic data. Here we set the dimension to be $d=2^{14}$ and the number of edge devices to be $n=20$. The local objective functions are of the form $f_i(x)=\frac{1}{2}(x-x_0)^\top A_i(x-x_0)$, where each $A_i\in\mathbb{R}^{d\times d}$ is a diagonal matrix, and we denote $A=\frac{1}{n}\sum_i A_i$. We generate $A_i$ such that the diagonal entries of $A$ is given by $A_{jj} = e^{-j/300}+0.001$ for each $j=1,\ldots,d$ while the diagonal of each $A_i$ is dense. We also let $\mathsf{g}_i(x)$ give approximately sparse stochastic gradients for every $x\in\mathbb{R}^d$. We refer to Appendix~\ref{appendix:numerical} for details on the test case.

We test three algorithms: the uncompressed vanilla SGD, Algorithm~\ref{alg:main}, and SGD with count sketch. The SGD with count sketch just replaces the gradient compression and reconstruction of Algorithm~\ref{alg:main} by the count sketch method \cite{charikar2002finding}.
We set $K=500$ for both Algorithm~\ref{alg:main} and SGD with count sketch. For Algorithm~\ref{alg:main}, we generate $\Phi$ from the WHT matrix, and uses the FFHT library~\cite{andoni2015practical} for fast WHT. We set $T=1000$, $\eta=1/\sqrt{T}$ and the initial point to be the origin for all three algorithms.

\begin{figure*}[htbp]
\centering
\subcaptionbox{Convergence of three algorithms ($w(t)=0$).
\label{fig:synthetic_convergence_3algs}}[.3\linewidth]
{
\includegraphics[width=\linewidth]{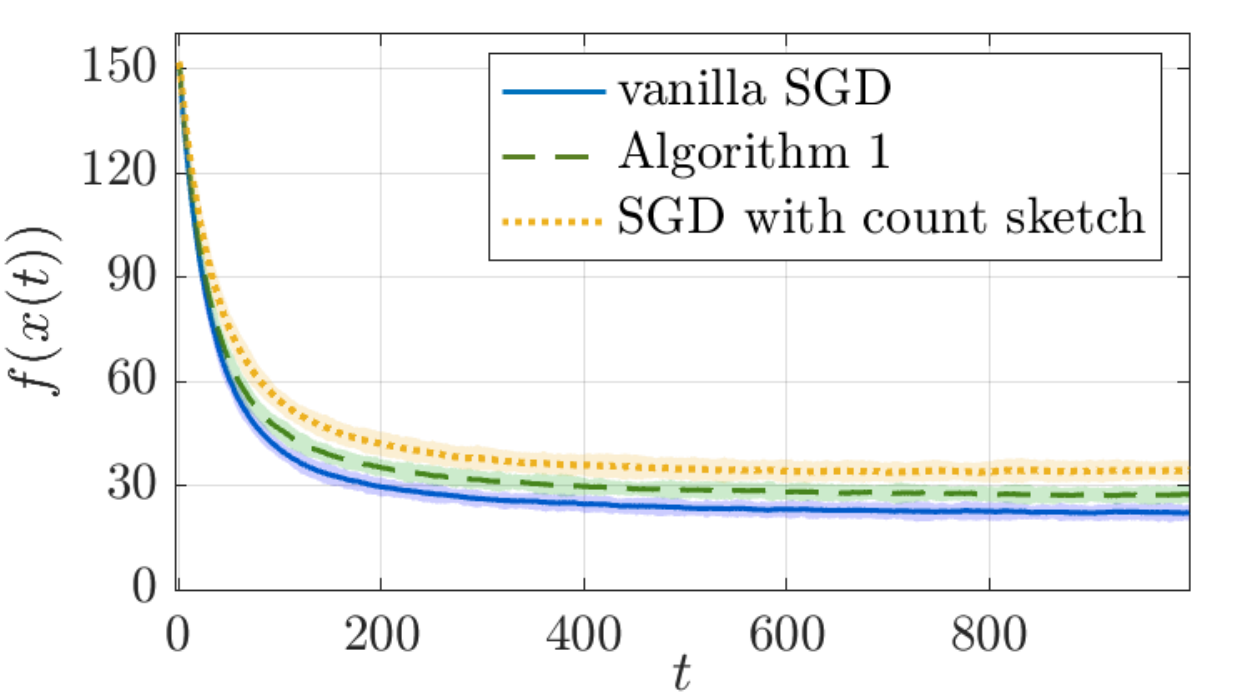}
}
\hspace{7pt}
\subcaptionbox{Evolution of $\operatorname{sp}(p(t))$ and $\operatorname{sp}(g(t))$ for Algorithm~\ref{alg:main} ($w(t)=0$).
\label{fig:synthetic_sparsity_level}}[.3\linewidth]
{
\includegraphics[width=\linewidth]{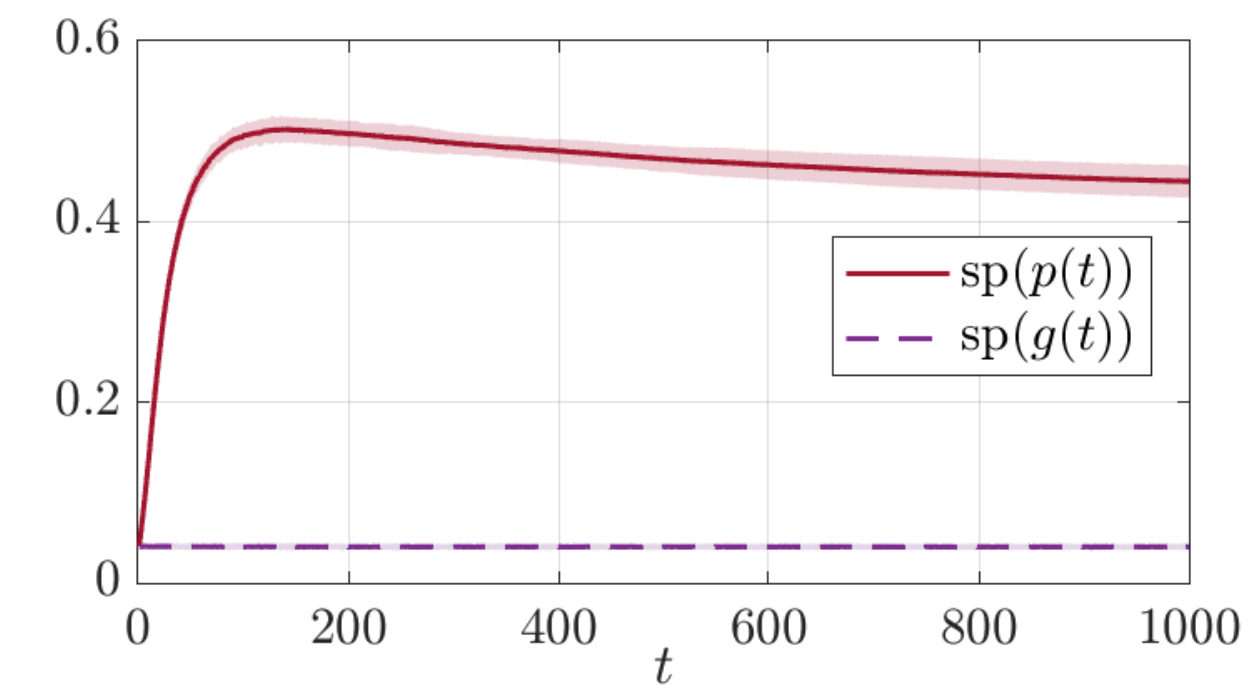}
}
\hspace{7pt}
\subcaptionbox{Convergence of Algorithm~\ref{alg:main} with different amplitudes of $w(t)$.\label{fig:synthetic_communication_noise}}[.3\linewidth]
{
\includegraphics[width=\linewidth]{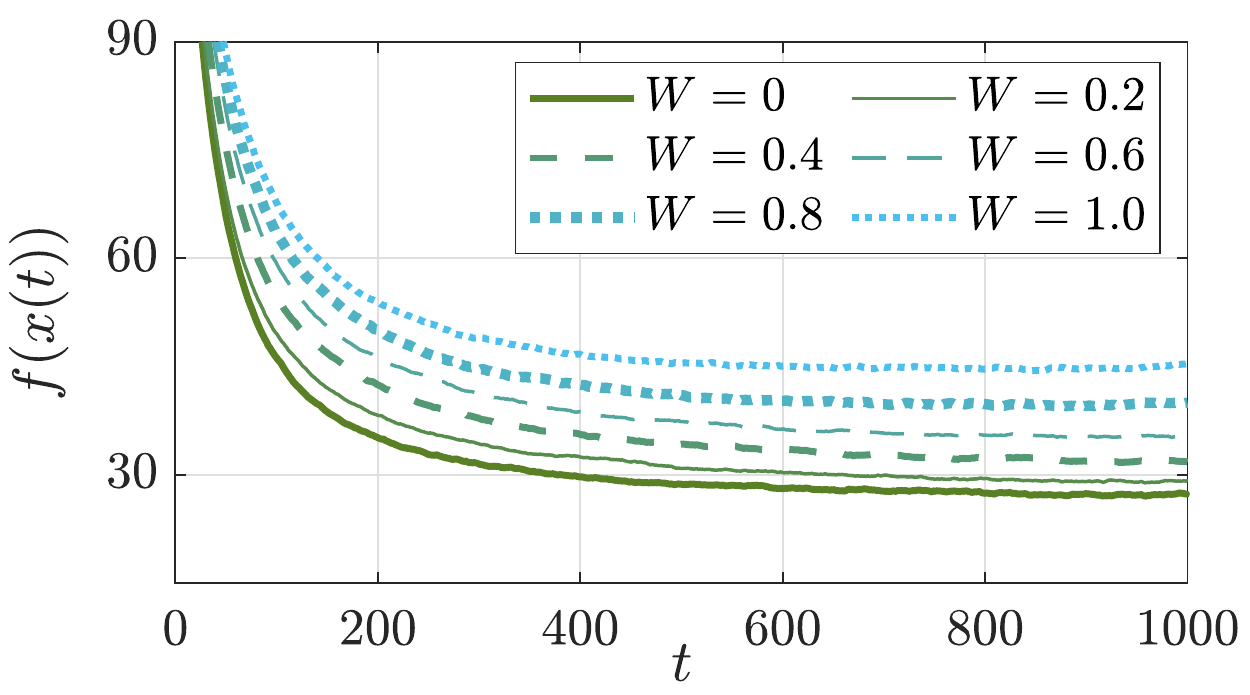}
}
\caption{Curves represent the average of $50$ random trials, and light-colored shades represent $3$-sigma confidence intervals.}
\vspace{-8pt}
\end{figure*}

Figure~\ref{fig:synthetic_convergence_3algs} illustrates the convergence of the three algorithms without communication channel error (i.e., $w(t)=0$). For Algorithm~\ref{alg:main}, we set $Q=5000$ (the compression rate $d/Q$ is $3.28$), and for SGD with count sketch we set the sketch size to be $16\times 500$ (the compression rate is $d/(16\times 500)=2.05$). We see that Algorithm~\ref{alg:main} has better convergence behavior while also achieves higher compression rate compared to SGD with count sketch. Our numerical experiments suggest that for approximately sparse signals, FIHT can achieve higher reconstruction accuracy and more aggressive compression than count sketch, and for signals that are not very sparse, FIHT also seems more robust.

Figure~\ref{fig:synthetic_sparsity_level} shows the evolution of $\operatorname{sp}(p(t))$ and $\operatorname{sp}(g(t))$ for Algorithm~\ref{alg:main}. We see that $\operatorname{sp}(p(t))$ is small for the first few iterations, and then increases and stabilizes around $0.5$, which suggests that the condition~\eqref{eq:sparsity_condition} is likely to have been violated for large $t$. On the other hand, Fig.~\ref{fig:synthetic_convergence_3algs} shows that Algorithm~\ref{alg:main} can still achieve relatively good convergence behavior.
This indicates a gap between the theoretical results in Section~\ref{subsec:theory} and the empirical results, and suggests our analysis could be improved. We leave relevant investigation as future work.

Figure~\ref{fig:synthetic_communication_noise} illustrates the convergence of Algorithm~\ref{alg:main} with different levels of communication channel noise. Here the entries of $w(t)$ are i.i.d. sampled from $\mathcal{N}(0,W^2)$ with $W\in\{0.2,0.4,0.6,0.8,1.0\}$. We see that the convergence of Algorithm~\ref{alg:main} gradually deteriorates as $W$ increases, suggesting its robustness against communication channel noise.

\subsection{Test Case of Federated Learning with CIFAR-10 Dataset}

We implement our algorithm on a residual network with 668426 trainable parameters in two different settings. We primarily use Tensorflow and MPI in the implementation of these results (details about the specific experimental setup can be found in Appendix \ref{appendix:numerical}). In addition, we shall only present upload compression results here; download compression is not considered to be as significant as the upload compression (given that download speeds are generally higher than upload speeds), and in our case the download compression rate is simply given by $d/K$. For both settings, we use the CIFAR10 dataset (60,000 32$\times$32$\times$3 images of 10 classes) with a 50,000/10,000 train/test split.

\begin{figure}[htbp]
\includegraphics[width=.72\linewidth]{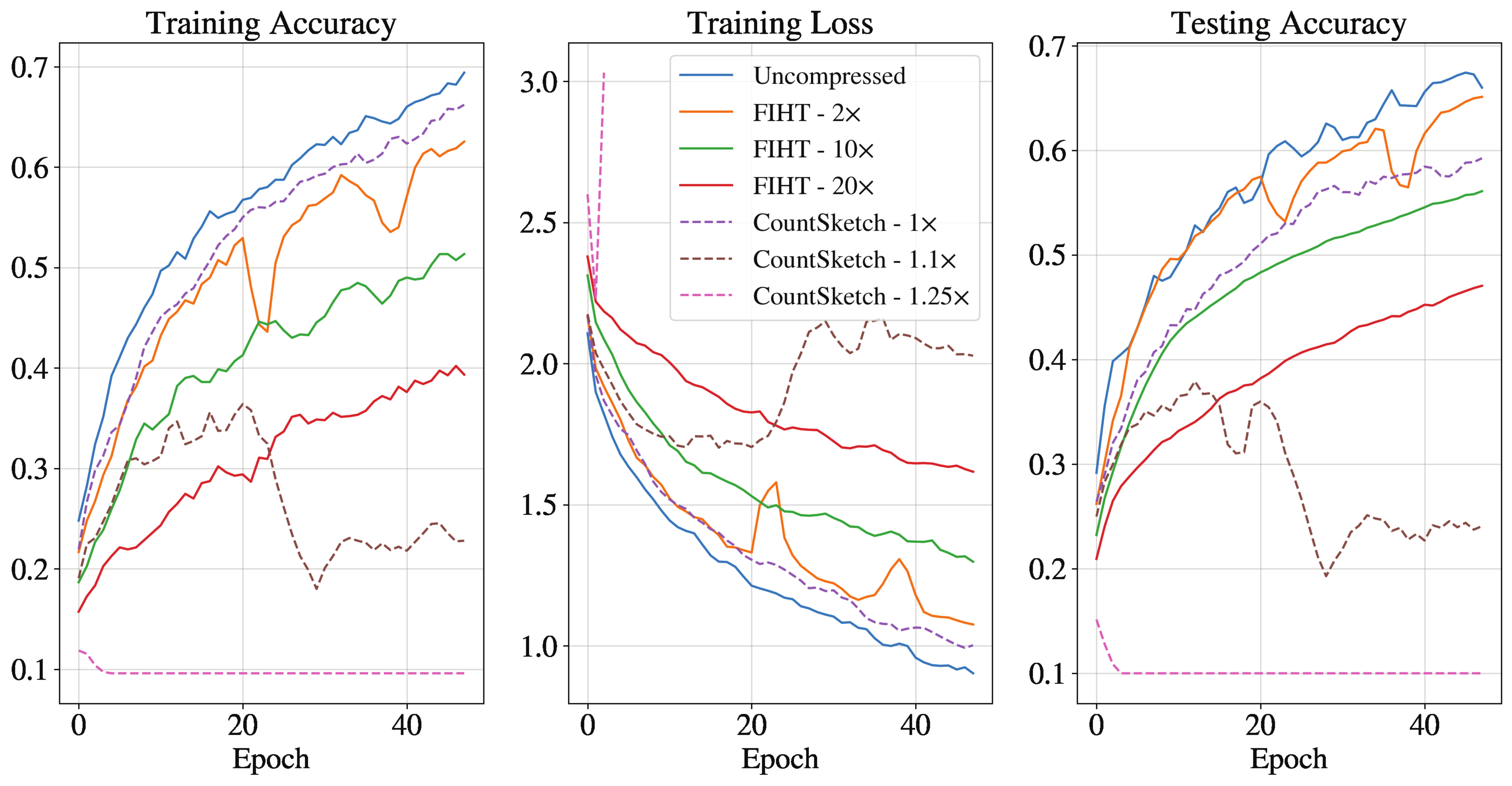}
\centering
\caption{Training Accuracy (left), Training Loss (middle), and Testing Accuracy (right) achieved on CIFAR10 with i.i.d. datasets for local workers.}
\label{fig:iid}
\end{figure}

In the first setting, we instantiate 100 workers and split the CIFAR10 training dataset such that all local datasets are i.i.d. As seen in Figure~\ref{fig:iid}, our algorithm is able to achieve 2$\times$ upload compression with marginal effect to the training and testing accuracy over 50 epochs. As the compression rate increases, the convergence of our method gradually deteriorates (echoing results in the synthetic case). For comparison, we also show the results of using Count Sketch with $5$ rows and $d/(5\lambda)$ columns, where $\lambda$ is the desired compression rate, in lieu of FIHT. In our setting, while uncompressed (1$\times$) Count Sketch performs well, it is very sensitive to higher compression, diverging for 1.1$\times$ and 1.25$\times$ compression.

\begin{figure}[htbp]
\includegraphics[width=.72\linewidth]{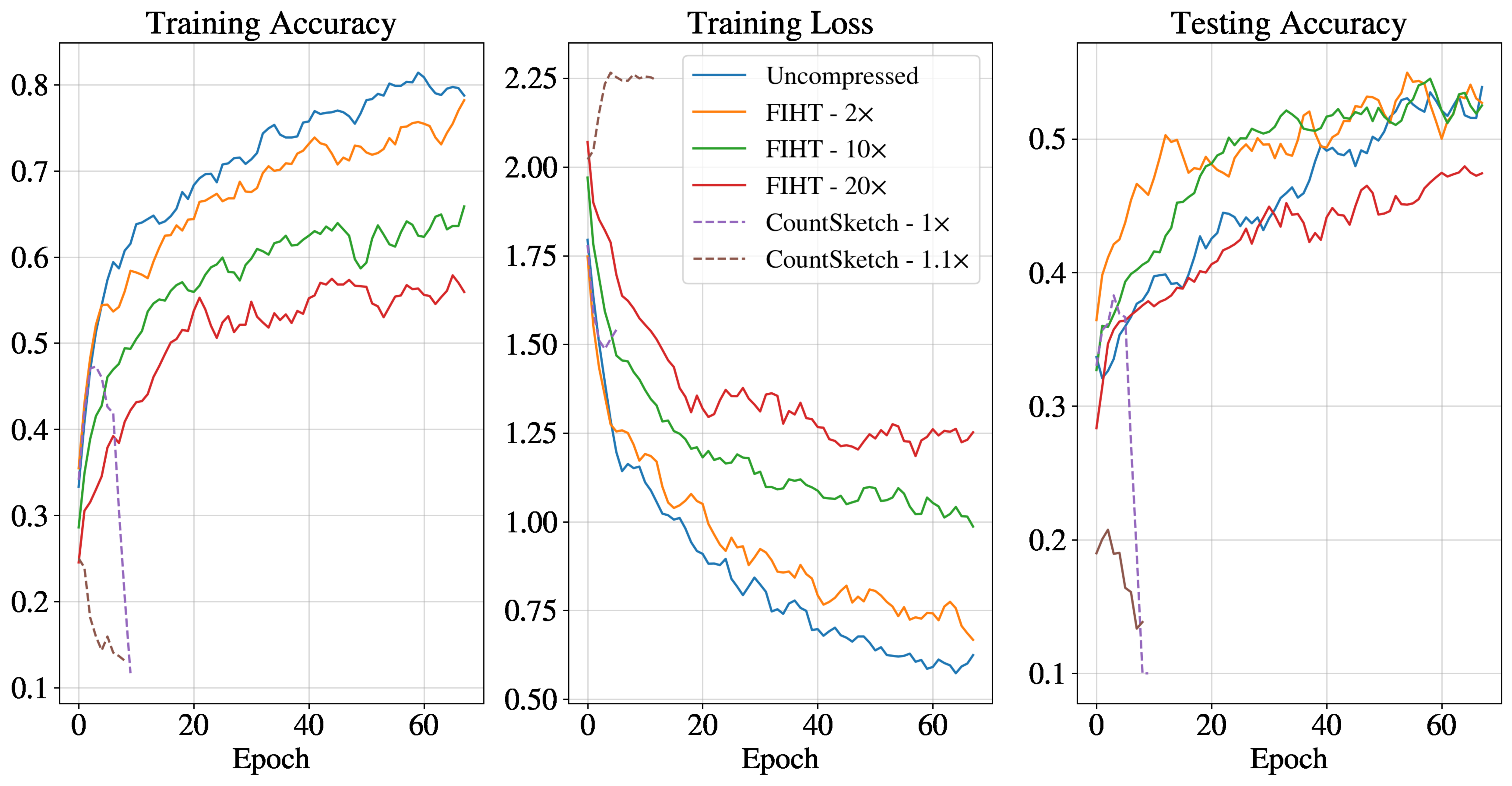}
\centering
\caption{Training Accuracy (left), Training Loss (middle), and Testing Accuracy (right) achieved on CIFAR10 with non-i.i.d. datasets for local workers.}
\label{fig:noniid}
\end{figure}

In the second setting, we split the CIFAR10 dataset into 10,000 sets of 5 images of a single class and assign these sets to 10,000 workers. Each epoch consists of 100 rounds with 1\% (100) workers participating in each round. In Figure~\ref{fig:noniid}, similar to the i.i.d. setting, we see that our algorithm's training accuracy convergence gradually deteriorates with higher compression. Typical of the non-i.i.d. setting, the testing accuracy is not as high as that of the i.i.d. setting. However, we note that FIHT is able to achieve 10$\times$ compression with negligible effect on the testing accuracy. In addition, Count Sketch with $5$ rows and $d/(5\lambda)$ columns diverges even for small compression rates in this problem setting.

\section{Conclusion}
In this paper, we develop a communication efficient SGD algorithm based on compressed sensing. This algorithm has several direct variants. For example, momentum method can be directly incorporated. Also, when the number of devices $n$ is very large, the server can choose to query compressed stochastic gradients from a randomly chosen subset of workers.

Our convergence guarantees require $\operatorname{sp}(p(t))$ to be persistently low, which is hard to check in practice. The numerical experiments also show that our algorithm can work even if $\operatorname{sp}(p(t))$ grows to a relatively high level. They suggest that our theoretical analysis can be further improved, which will be an interesting future direction.

%%%%%%%%%%%%%%%%%%%%%%%%%%%%%%%%%%%%%%%%%%%%%%%%%%%%%%%%%%%%%%%%%%%%%%%%%%%%%%%%

%%%%%%%%%%%%%%%%%%%%%%%%%%%%%%%%%%%%%%%%%%%%%%%%%%%%%%%%%%%%%%%%%%%%%%%%%%%%%%%%

%%%%%%%%%%%%%%%%%%%%%%%%%%%%%%%%%%%%%%%%%%%%%%%%%%%%%%%%%%%%%%%%%%%%%%%%%%%%%%%%

\appendix

\section{Auxiliary Results}

In this section, we provide some auxiliary results for the proof of Theorem~\ref{theorem:main}. We first give an alternative form of the reconstruction error derived from the condition~\eqref{eq:sparsity_condition} and the performance guarantee~\eqref{eq:FIHT_performance}.

\vspace{-2pt}
\begin{lemma}\label{lemma:compressor}
Suppose the conditions in Lemma~\ref{lemma:FIHT_performance} are satisfied. Let $w\in\mathbb{R}^Q$ be arbitrary, and let $x\in\mathbb{R}^d$ satisfy that $\operatorname{sp}(x)$ is upper bounded by the right-hand side of~\eqref{eq:sparsity_condition}. Then
$$
\|\mathcal{A}(\Phi x+w;\Phi)-x\|_2
\leq \sqrt{\frac{\gamma}{2}} \|x\|_2
+C_{\mathcal{A},\mathrm{n}}\|w\|_2.
$$
\end{lemma}
\begin{proof}
By \cite[Lemma 7]{gilbert2007one}, we have $\|x-x^{[K]}\|_2\leq \|x\|_1/(2\sqrt{K})$. Therefore by Lemma~\ref{lemma:FIHT_performance},
\begin{align*}
\|\mathcal{A}(\Phi x+w;\Phi)-x\|_2
\leq\ &
\frac{C_{\mathcal{A},\mathrm{s}}+1}{2\sqrt{K}}\|x\|_1
+\frac{C_{\mathcal{A},\mathrm{s}}}{\sqrt{K}}\big\|x-x^{[K]}\big\|_1
+C_{\mathcal{A},\mathrm{n}}\|w\|_2 \\
\leq\ &
\left[\frac{C_{\mathcal{A},\mathrm{s}}+1}{2}
+C_{\mathcal{A},\mathrm{s}}\!\left(1-\mfrac{K}{d}\right)\right]\frac{\big\|x\big\|_1}{\sqrt{K}}
+C_{\mathcal{A},\mathrm{n}}\|w\|_2 %\\
% =\ &
% \frac{1}{2}\left[1
% +C_{\mathcal{A},\mathrm{s}}\!\left(3-\mfrac{2K}{d}\right)\right]
% \sqrt{\mfrac{d}{K}}\sqrt{\operatorname{sp}(x)}
% +C_{\mathcal{A},\mathrm{n}}\|w\|_2
.
\end{align*}
Plugging in the definition and upper bound of $\operatorname{sp}(x)$ leads to the result.
\end{proof}

Next, we derive a bound on the second moment of $e(t)$.
\begin{lemma}\label{lemma:err_bound}
We have
\begin{align*}
\frac{1}{T}\sum_{t=1}^{T}
\mathbb{E}\!\left[\|e(t)\|_2^2\right]\!
\leq\,&
\frac{2\eta^2}{1\!-\!\gamma}
\bigg[
\frac{\gamma(1\!+\!\gamma)}{1\!-\!\gamma}G^2
\!+\! 2\!\left(C_{\mathcal{A},\mathrm{n}}
\!\!+\!\!\sqrt{Q/d}\right)^{\!2}\!\!\sigma^2
\bigg]
+\frac{2\eta^2\gamma(1\!+\!\gamma)}{(1\!-\!\gamma)^2}\cdot\frac{1}{T}\sum_{t=1}^T\mathbb{E}\!\left[\|\nabla f(x(t))\|_2^2\right].
\end{align*}
\end{lemma}
\begin{proof}
By definition, we have
\begin{align*}
\mathbb{E}\!\left[\|e(t\!+\!1)\|_2^2\right]
\leq \ &
\mathbb{E}\Big[\big(\|\Delta(t)-p(t)\|_2
+\mfrac{\eta Q}{d}\big\|\Phi^\top w(t)\big\|_2\big)^2\Big] \\
=\ &
\mathbb{E}\Big[
\big(\left\|\mathcal{A}(\Phi p(t)+ \eta w(t))-p(t)\right\|_2
+\eta\sqrt{Q/d}\|w\|_2\big)^2\Big] \\
\leq\ &
\mathbb{E}\Big[
\big(
\sqrt{\gamma/2}\|p(t)\|_2
+\eta\big(C_{\mathcal{A},\mathrm{n}}\!+\!\sqrt{Q/d}\big)\|w\|_2
\big)^2\Big] \\
\leq\ &
\gamma
\mathbb{E}\!\left[\|p(t)\|_2^2\right] + 2\eta^2\big(C_{\mathcal{A},\mathrm{n}}\!+\!\sqrt{Q/d}\big)^2
\mathbb{E}\!\left[\| w(t)\|_2^2\right] \\
\leq\ &
\gamma
\mathbb{E}\!\left[\| \eta g(t)+e(t)\|_2^2\right] + 2\eta^2 \big(C_{\mathcal{A},\mathrm{n}}\!+\!\sqrt{Q/d}\big)^2 \sigma^2,
\end{align*}
where the first inequality follows from Lemma~\ref{lemma:compressor}, and the second inequality follows from the definition of $p(t)$ and the assumption that $\mathbb{E}[\|w\|_2^2]\leq\sigma^2$. Notice that
\begin{align*}
\mathbb{E}\!\left[\|\eta g(t)+e(t)\|_2^2\right]
\leq
\left(1+\frac{2\gamma}{1-\gamma}\right)
\mathbb{E}\!\left[\|\eta g(t)\|_2^2\right]
+\left(1+\frac{1-\gamma}{2\gamma}\right)
\mathbb{E}\!\left[\|e(t)\|_2^2\right] %\\
%=\ &
%\mfrac{1+\gamma}{2\gamma}\mathbb{E}\!\left[\|e(t)\|_2^2\right]
%+\eta^2\mfrac{1+\gamma}{1-\gamma}
%\mathbb{E}\!\left[
%\left\|g(t)\right\|_2^2
%\right]
,
\end{align*}
which leads to
\begin{align*}
\mathbb{E}\!\left[\|e(t\!+\!1)\|_2^2\right]
\leq
\frac{1 \!+\! \gamma}{2}\mathbb{E}\!\left[\|e(t)\|_2^2\right]
+\eta^2\frac{\gamma(1 \!+\! \gamma)}{1 \!-\! \gamma}
\mathbb{E}\!\left[\|g(t)\|_2^2\right]
+2\eta^2\!\big(C_{\mathcal{A},\mathrm{n}}
\!+\!\sqrt{Q/d}\big)^{\!2}\! \sigma^2.
\end{align*}
By $\mathbb{E}[g(t)|x(t)]=\nabla f(x(t))$ and Assumption~\ref{assumption:var_grad}, we have
\begin{align*}
\mathbb{E}\!\left[
\|g(t)\|_2^2
\right]
=\ & \mathbb{E}\!\left[\|\nabla f(x(t))\|_2^2\right]
+\mathbb{E}\!\left[\|g(t)-\nabla f(x(t))\|_2^2\right] \\
\leq\ & \mathbb{E}\!\left[\|\nabla f(x(t))\|_2^2\right]
+ G^2.
\end{align*}
Therefore
\begin{align*}
\mathbb{E}\!\left[\|e(t\!+\!1)\|_2^2\right]
\leq
\frac{1 \!+\! \gamma}{2}
\mathbb{E}\mkern-4mu\left[\|e(t)\|_2^2\right]
\!+\! \eta^2\mkern-1mu
\frac{\gamma(1 \!+\! \gamma)}{1 \!-\! \gamma}\mathbb{E}\mkern-4mu\left[\|\nabla \!f(x(t)\mkern-2mu)\|_2^2\right]
+
\eta^2
\!\left[
\frac{\gamma(1 \!+\! \gamma)}{1 \!-\! \gamma}G^2
\!+\! 2\big(C_{\mathcal{A},\mathrm{n}}
\!\!+\!\!\sqrt{Q/d}\big)^{\!2}\!\sigma^2
\right].
\end{align*}
By summing over $t=1,\ldots,T$ and noting that $e(1)=0$ and $\mathbb{E}[\|e(T+1)\|_2^2]\geq 0$, we get
\begin{align*}
\frac{1}{T}\sum_{t=1}^{T}
\mathbb{E}\!\left[\|e(t)\|_2^2\right]
\leq\ &
\frac{1\!+\!\gamma}{2T}
\sum_{t=1}^{T}
\mathbb{E}\!\left[\|e(t)\|_2^2\right]
+
\frac{\eta^2\gamma(1 \!+\!\gamma)}{(1\!-\!\gamma)T}
\sum_{t=1}^T
\mathbb{E}\!\left[\|\nabla f(x(t)\mkern-2mu)\|_2^2\right] \\
& +\eta^2
\!\left[
\frac{\gamma(1 \!+\! \gamma)}{1 \!-\!\gamma}G^2
+ 2\big(C_{\mathcal{A},\mathrm{n}}
\!\!+\!\!\sqrt{Q/d}\big)^{\!2}\!\sigma^2
\right],
\end{align*}
which then leads to the desired result.
\end{proof}

\section{Proof of Theorem~\ref{theorem:main}}\label{appendix:proof}

\noindent\textbf{Convex case:} Denote $\tilde{x}(t)=x(t)-e(t)$, and it can be checked that
$$
\begin{aligned}
\tilde{x}(t\!+\!1)
& =x(t\!+\!1) - e(t\!+\!1)=x(t) - \Delta(t)
-(p(t)-\Delta(t)) = x(t)-p(t) \\
& = \tilde{x}(t) - \eta g(t).
\end{aligned}
$$
We then have
$$
\|\tilde{x}(t\!+\!1)-x^\ast\|_2^2
=\|\tilde{x}(t)-x^\ast\|_2^2
+\eta^2\|g(t)\|_2^2
-2\eta\langle g(t),\tilde{x}(t)-x^\ast\rangle.
$$
By taking the expectation and noting $\mathbb{E}[g(t)|x(t)]=\nabla f(x(t))$ and Assumption~\ref{assumption:var_grad}, we get
$$
\begin{aligned}
\mathbb{E}\!\left[\|\tilde{x}(t\!+\!1)-x^\ast\|_2^2\right]
\leq\ &
\mathbb{E}\!\left[\|\tilde{x}(t)-x^\ast\|_2^2\right]
+\eta^2\!\left(
\mathbb{E}\!\left[\|\nabla f(x(t))\|_2^2\right]
+G^2\right) \\
&
-2\eta\,\mathbb{E}\!\left[
\langle \nabla f(x(t)),x(t)-x^\ast\rangle\right]
+2\eta\,\mathbb{E}\!\left[
\langle \nabla f(x(t)),e(t)\rangle\right],
\end{aligned}
$$
which leads to
$$
\begin{aligned}
\mathbb{E}\!\left[
\langle \nabla f(x(t)),x(t)-x^\ast\rangle\right]
\leq\ &
\frac{1}{2\eta}(\mathbb{E}\!\left[\|\tilde{x}(t)-x^\ast\|_2^2\right]
-\mathbb{E}\!\left[\|\tilde{x}(t\!+\!1)-x^\ast\|_2^2\right])
+\frac{\eta G^2}{2} \\
&
+\frac{\eta}{2}\,\mathbb{E}\!\left[\|\nabla f(x(t))\|_2^2\right]
+\mathbb{E}\!\left[
\langle \nabla f(x(t)),e(t)\rangle\right] \\
\leq\ &
\frac{1}{2\eta}(\mathbb{E}\!\left[\|\tilde{x}(t)-x^\ast\|_2^2\right]
-\mathbb{E}\!\left[\|\tilde{x}(t\!+\!1)-x^\ast\|_2^2\right])
+\frac{\eta G^2}{2} \\
&
+\frac{\eta+(3L)^{-1}}{2}\,\mathbb{E}\!\left[\|\nabla f(x(t))\|_2^2\right]
+\frac{3L}{2}\mathbb{E}\!\left[ \|e(t)\|_2^2\right],
\end{aligned}
$$
where in the second inequality we used $\langle \nabla f(x(t)),e(t)\rangle
\leq \frac{1}{6L}\|\nabla f(x(t))\|_2^2
+\frac{3L}{2}\|e(t)\|_2^2$. Now, we take the average of both sides over $t=1,\ldots,T$ and plug in the bound in Lemma~\ref{lemma:err_bound} to get
\begin{align*}
\frac{1}{T}\sum_{t=1}^T
\mathbb{E}\!\left[
\langle \nabla f(x(t)),x(t)-x^\ast\rangle\right]
\leq\ &
\frac{1}{2\eta T}\|x(1)-x^\ast\|_2^2
+\frac{\eta G^2}{2}
+\frac{3\eta^2 L}{1\!-\!\gamma}\!\left[
\frac{\gamma(1\!+\!\gamma)}{1\!-\!\gamma}G^2
+ 2\big(C_{\mathcal{A},\mathrm{n}}
\!+\!\sqrt{Q/d}\big)^{\!2}\sigma^2
\right] \\
&
+
\left(\frac{\eta\!+\!(3L)^{-1}}{2}
\!+\!\frac{3\eta^2 L\gamma(1\!+\!\gamma)}{(1\!-\!\gamma)^2}\right)
\frac{1}{T}\sum_{t=1}^T\mathbb{E}\!\left[\|\nabla f(x(t))\|_2^2\right]\!.
\end{align*}
By $\eta=1/(L\sqrt{T})$, we can show that for sufficiently large $T$,
$$
\frac{\eta+(3L)^{-1}}{2}
+\frac{3\eta^2 L\gamma(1+\gamma)}{(1-\gamma)^2}
= \frac{1}{L}
\left(\frac{1}{6}
+\frac{1}{2\sqrt{T}}
+\frac{3\gamma(1+\gamma)}{T(1-\gamma)^2}\right)
\leq \frac{1}{4L}.
$$
Thus
$$
\begin{aligned}
\frac{1}{T}\sum_{t=1}^T
\mathbb{E}\!\left[
\langle \nabla f(x(t)),x(t)-x^\ast\rangle\right]
\leq\ &
\frac{1}{2\eta T}\|x(1)-x^\ast\|_2^2
+\frac{\eta G^2}{2}
+\frac{3\eta^2 L}{1\!-\!\gamma}\!\left[
\frac{\gamma(1\!+\!\gamma)}{1\!-\!\gamma}G^2
+ 2\big(C_{\mathcal{A},\mathrm{n}}
\!+\!\sqrt{Q/d}\big)^{\!2}\sigma^2
\right] \\
&
+
\frac{1}{4L}\cdot
\frac{1}{T}\sum_{t=1}^T\mathbb{E}\!\left[\|\nabla f(x(t))\|_2^2\right].
\end{aligned}
$$
By the convexity of $f$, we have
$f(x(t))-f(x^\ast)\leq \langle \nabla f(x(t)),x(t)-x^\ast\rangle$. Furthermore, since $f$ is $L$-smooth, we see that
$$
f(x^\ast)
\leq f\!\left(x(t)-\frac{1}{L}\nabla f(x(t))\right)
\leq f(x(t))
-\frac{1}{2L}\|\nabla f(x(t))\|_2^2,
$$
which leads to $\|\nabla f(x(t))\|_2^2
\leq 2L(f(x(t))-f(x^\ast))$. We then get
$$
\begin{aligned}
\frac{1}{T}\sum_{t=1}^T
\mathbb{E}\!\left[f(x(t))-f(x^\ast)\right]
\leq\ &
\frac{1}{2T}\sum_{t=1}^T\mathbb{E}\!\left[f(x(t))-f(x^\ast)\right]
+\frac{1}{2\eta T}\|x(1)-x^\ast\|_2^2
+\frac{\eta G^2}{2} \\
& +\frac{3\eta^2 L}{1\!-\!\gamma}\!\left[
\frac{\gamma(1\!+\!\gamma)}{1\!-\!\gamma}G^2
+ 2\big(C_{\mathcal{A},\mathrm{n}}
\!+\!\sqrt{Q/d}\big)^{\!2}\sigma^2
\right].
\end{aligned}
$$
By subtracting $\frac{1}{2T}
\sum_{t=1}^T\mathbb{E}\!\left[f(x(t)) \!-\! f(x^\ast)\right]$ from both sides of the inequality, and using the bound $f(\bar{x}(t))
\leq \frac{1}{T}\sum_{t=1}^T f(x(t))$ that follows from the convexity of $f$, we get the final bound.

\vspace{5pt}
\noindent\textbf{Nonconvex case:} Denote $\tilde{x}(t) = x(t) - e(t)$, and it can be checked that
$
\tilde{x}(t\!+\!1)
= \tilde{x}(t) - \eta g(t)
$.
Since $f$ is $L$-smooth, we get
$$
f(\tilde{x}(t\!+\!1))
\leq f(\tilde{x}(t))
-
\eta\langle\nabla f(\tilde{x}(t)),
g(t)\rangle
+\frac{\eta^2 L}{2}
\|g(t)\|_2^2.
$$
By taking the expectation and using $\mathbb{E}[g(t)|x(t)]
=\nabla f(x(t))$ and Assumption~\ref{assumption:var_grad}, we see that
$$
\begin{aligned}
\mathbb{E}\!\left[f(\tilde{x}(t\!+\!1))\right]
\leq\ &
\mathbb{E}\!\left[f(\tilde{x}(t))\right]
-\eta\,
\mathbb{E}\!\left[\langle\nabla f(\tilde{x}(t)),\nabla f(x(t))\rangle\right]
+\frac{\eta^2 L}{2}
\!\left(\mathbb{E}\!\left[\|\nabla f(x(t))\|_2^2\right]+G^2\right) \\
=\ &
\mathbb{E}\!\left[f(\tilde{x}(t))\right]
-\eta\,
\mathbb{E}\!\left[\|\nabla f(x(t))\|_2^2\right] \\
& -\eta\,
\mathbb{E}\!\left[\langle\nabla f(\tilde{x}(t))-\nabla f(x(t)),\nabla f(x(t))\rangle\right]
+\frac{\eta^2 L}{2}\!\left(\mathbb{E}\!\left[\|\nabla f(x(t))\|_2^2\right]+G^2\right) \\
\leq\ &
\mathbb{E}\!\left[f(\tilde{x}(t))\right]
-\frac{\eta(1-\eta L)}{2}
\mathbb{E}\!\left[\|\nabla f(x(t))\|_2^2\right] \\
&
+\frac{\eta}{2}\,
\mathbb{E}\!\left[
\|\nabla f(\tilde{x}(t))-\nabla f(x(t))\|_2^2\right]
+\frac{\eta^2 L}{2} G^2 \\
\leq\ &
\mathbb{E}\!\left[f(\tilde{x}(t))\right]
-\frac{\eta(1-\eta L)}{2}
\mathbb{E}\!\left[\|\nabla f(x(t))\|_2^2\right]
+\frac{\eta L^2}{2}
\mathbb{E}\!\left[
\|e(t)\|_2^2\right]
+\frac{\eta^2 L}{2} G^2,
\end{aligned}
$$
where in the second inequality we used $\langle\nabla f(\tilde{x}(t))-\nabla f(x(t)),\nabla f(x(t))\rangle\leq \frac{1}{2}\|\nabla f(\tilde{x}(t))-\nabla f(x(t))\|_2^2+\frac{1}{2}\|\nabla f(x(t))\|_2^2$, and in the last inequality we used the $L$-smoothness of $f$.
By taking the telescoping sum, we get
$$
\begin{aligned}
\frac{1-\eta L}{2}
\frac{1}{T}\sum_{t=1}^T\mathbb{E}\left[\|\nabla f(x(t))\|_2^2\right]
\leq\ &
\frac{f(x(1))-f^\ast}{\eta}
+\frac{\eta LG^2}{2}
+\frac{L^2}{2}\cdot\frac{1}{T}\sum_{t=1}^T
\mathbb{E}\!\left[\|e(t)\|_2^2\right].
\end{aligned}
$$
After plugging in the bound in Lemma~\ref{lemma:err_bound}, we get
$$
\begin{aligned}
\frac{1-\eta L}{2}
\frac{1}{T}\sum_{t=1}^T\mathbb{E}\left[\|\nabla f(x(t))\|_2^2\right]
\leq\ &
\frac{f(x(1))-f^\ast}{\eta}
+\frac{\eta LG^2}{2}
+\frac{\eta^2 L^2}{1\!-\!\gamma}\!\left[
\frac{\gamma(1\!+\!\gamma)}{1\!-\!\gamma}G^2
+ 2\!\left(C_{\mathcal{A},\mathrm{n}}
\!+\!\sqrt{Q/d}\right)^{\!2}\!\sigma^2
\right] \\
&
+\eta^2 L^2 \frac{\gamma(1+\gamma)}{(1-\gamma)^2}
\cdot\frac{1}{T}\sum_{t=1}^T
\mathbb{E}\!\left[\|\nabla f(x(t))\|_2^2\right].
\end{aligned}
$$
By choosing $\eta = 1/(L\sqrt{T})$ and letting $T$ be sufficiently large, we have
$$
\frac{1-\eta L}{2}
-\eta^2 L^2\frac{\gamma(1+\gamma)}{(1-\gamma)^2}
=\frac{1}{2}-\frac{1}{2\sqrt{T}}
-\frac{\gamma(1+\gamma)}{T(1-\gamma)^2}
\geq \frac{1}{3}.
$$
Finally, we get
$$
\frac{1}{3T}\sum_{t=1}^T
\mathbb{E}\!\left[\|\nabla f(x(t))\|_2^2\right]
\leq
\frac{2L\!\left(f(x(1)) \!-\! f^\ast\right) + G^2}{2\sqrt{T}}
+\frac{1}{T}\!\left[
\frac{\gamma(1+\gamma)}{(1-\gamma)^2}G^2
+\frac{ 2\big(C_{\mathcal{A},\mathrm{n}}
\!+\!\sqrt{Q/d}\big)^{\!2}\sigma^2}{1-\gamma}
\right]
\!,
$$
which completes the proof.

\section{Further Details on Numerical Experiments}\label{appendix:numerical}

For all numerical experiments in this paper, we employ the WHT matrix as the ``base matrix'' for generating the sensing matrix $\Phi$. We choose WHT instead of DCT here because the library \texttt{Fast Fast Hadamard Transform (FFHT)}~\cite{andoni2015practical} provides a heavily optimized C99 implementation of fast WHT, which turns out to be faster than the fast DCT in the Python library \texttt{SciPy} on our computing platforms.

Note that the WHT matrix is only defined for $d$ that is a power of $2$. In order to use WHT for general $d$, we let $d_{\mathrm{aug}}=2^{\lceil\log_2 d\rceil}$ and generate $\Phi\in\mathbb{R}^{Q\times d_{\mathrm{aug}}}$ by randomly choosing Q rows from the $d_{\mathrm{aug}}\times d_{\mathrm{aug}}$ WHT matrix as in Proposition~\ref{prop:subsampled_Fourier}. Then for any $u\in\mathbb{R}^d$, we can first pad zeros to the vector $u$ to augment its dimension to be $d_{\mathrm{aug}}$, and then multiply $\Phi\in\mathbb{R}^{Q\times d_{\mathrm{aug}}}$ with the augmented vector to form the compressed vector of dimension $Q$. For the reconstruction, after obtaining the recovered signal, we can truncate the last $d_{\mathrm{aug}}\!-\!d$ entries to get a vector of the original dimension $d$.

Code for the federated learning test case is available at \url{https://github.com/vikr53/cosamp}. 

\subsection{Test Case with Synthetic Data}

All experiments on the synthetic test case were run on a MacBook Pro Model A1502 with 2.9 GHz Dual-Core Intel Core i5 CPU and 16 GB memory. GPU acceleration was not exploited. We implement count sketch by our own code for the synthetic test case.

Recall that the local objective functions are given by $f_i(x)=\frac{1}{2}(x-x_0)^\top A_i(x-x_0)$, where each $A_i$ is a diagonal matrix. The $j$'th diagonal entries of $A_1,\ldots,A_n$ are randomly sampled such that $\big((A_1)_{jj},(A_2)_{jj},\ldots,(A_n)_{jj}\big)$ satisfies the Gaussian distribution $\mathcal{N}((e^{-j/300}+0.001)\mathbf{1},I-\frac{1}{n}\mathbf{1}\mathbf{1}^\top)$, where $I$ denotes the identity matrix and $\mathbf{1}$ denotes the vector of all ones. As a result, $A\coloneqq \frac{1}{n}\sum_i A_i$ has diagonal entries given by $A_{jj}=e^{-j/300}+0.001$. The point $x_0\in\mathbb{R}^d$ is randomly generated from the standard Gaussian distribution. We fix each $A_i$ and $x_0$ once they have been generated.

The stochastic gradient $\mathsf{g}_i(x)$ will be given by
$$
\mathbf{g}_i(x)=A_i(x-x_0)+
R_1\cdot A\mathsf{u}_1+R_2\cdot (\mathsf{b}\odot \mathsf{u}_2)
$$
Here $\mathsf{u}_1$ and $\mathsf{u}_2$ are i.i.d. random vectors following the standard Gaussian distribution $\mathcal{N}(0,I)$; the entries of $\mathsf{b}\in\mathbb{R}^d$ are i.i.d. satisfying the Bernoulli distribution $\mathrm{Ber}(p_{\mathsf{b}})$ with $p_{\mathsf{b}}\in(0,1)$; $\odot$ denotes the Hadamard product. We set $R_1=12.5, R_2=50, p_{\mathsf{b}}=1.5\times 10^{-3}$ in our test case.

\subsection{Test Case of Federated Learning with CIFAR-10 Dataset}

We ran all federated learning experiments using the research computing service provided by the authors' institution, and we did not exploit GPU acceleration. We primarily use \texttt{Tensorflow}~\cite{tensorflow2015-whitepaper} as the deep learning framework. We use \texttt{MPI for Python}~\cite{DALCIN2008655} for parallel computing on multiple processors, and communication between nodes is facilitated by \texttt{Open MPI}~\cite{gabriel04:_open_mpi}. Additionally, we use the \texttt{CSVec} package (available at \url{https://github.com/nikitaivkin/csh}) for count sketch.

Regarding the compressed sensing algorithms, we set $K=30000$ (number of nonzero entries in the reconstructed signal) for both FIHT and count sketch. The dimension $Q$ of the compressed gradients for FIHT is simply $Q=d/\lambda$ where $\lambda$ is the desired compression rate. The sketch size for count sketch is $r$~rows $\times$  $\mfrac{d}{r\lambda}$~columns; we tried $r\in\{1,5,10,50\}$ and found the best results with $r=5$ for all tested compression rates.

Recall that we conduct experiments on the i.i.d. setting and the non-i.i.d. setting. The differences between the i.i.d. setting and the non-i.i.d. settings include the following:
\begin{enumerate}[leftmargin=15pt]
\item In the i.i.d. setting, there are $100$ local workers, and we split the training dataset evenly into $100$ local datasets, so that each local dataset has exactly $50$ samples for each of the $10$ classes. In the non-i.i.d. setting, there are $10000$ workers, and each of the $10000$ local datasets has $5$ samples belonging to the same class.
\item In the i.i.d. setting, the server queries compressed local gradients from all workers for each iteration. In the non-i.i.d. setting, the server queries compressed local gradients from $1\%$ of all workers for each iteration.
\item In the i.i.d. setting, we set the local batch size to be $8$. In the non-i.i.d. setting, we set the local batch size to be $1$.
\item In the non-i.i.d. setting, we perform data augmentation (while dividing the dataset amongst the 10,000 workers) using random flips and random rotations.
\end{enumerate}
The step size $\eta$ of both the i.i.d. and the non-i.i.d. settings are set to be $0.01$. We also use the same parameters of FIHT and count sketch for both the i.i.d. and the non-i.i.d. settings.

\begin{comment}
\begin{table}[h]
\begin{center}
\begin{tabular}{cccc}
    \toprule
    \textbf{Setup} & \textbf{Step size $\eta$} & \textbf{Batch size} & \textbf{\# of nonzero entries $K$} \\
    \midrule
    i.i.d. & 0.01 & 8 & 30000 \\
    non-i.i.d. & 0.01 & 1 & 30000 \\
    \bottomrule
\end{tabular}
\end{center}
\caption{Hyperparameters for the federated learning test case.}
\label{tab:hyperparam}
\end{table}
\end{comment}

\paragraph{Experiments on the reconstruction error:} We also conducted numerical experiments that compare the reconstruction errors of FIHT and count sketch. Here we let $g\in\mathbb{R}^d$ with $d=668426$ be randomly generated by
$
g = g_{\mathrm{sp}}+ g_{\mathrm{n}}
$,
where $g_{\mathrm{sp}}$ is a vector having $30000$ nonzero entries that are i.i.d. sampled from the standard Gaussian distribution, and $g_{\mathrm{n}}\sim\mathcal{N}(0,0.05^2 I)$. We apply FIHT and count sketch to the compression and reconstruction of $g$, and test the relative reconstruction error incurred at different compression rates $\lambda$. The relative reconstruction error is defined as $\|g-\hat{g}\|_2^2/\|g\|_2^2$ where $\hat{g}$ is the reconstructed signal by FIHT or count sketch. For a given compression rate $\lambda$, we set $Q=d/\lambda$ for FIHT, and set the sketch size to be $5$~rows $\times$ $\mfrac{d}{5\lambda}$~columns for count sketch. We set $\|\hat{g}\|_0=K=30000$ for both FIHT and count sketch. For each compression rate $\lambda$, we test FIHT and count sketch over $20$ random instances of $g$; we fix the sensing matrix $\Phi$ of FIHT and the sketching operator of count sketch for the $20$ random trials. Figure~\ref{fig:recons} empirically confirms that the reconstruction performance of count sketch is worse than FIHT, especially at higher compression rates.

\begin{figure}
\includegraphics[width=.4\linewidth]{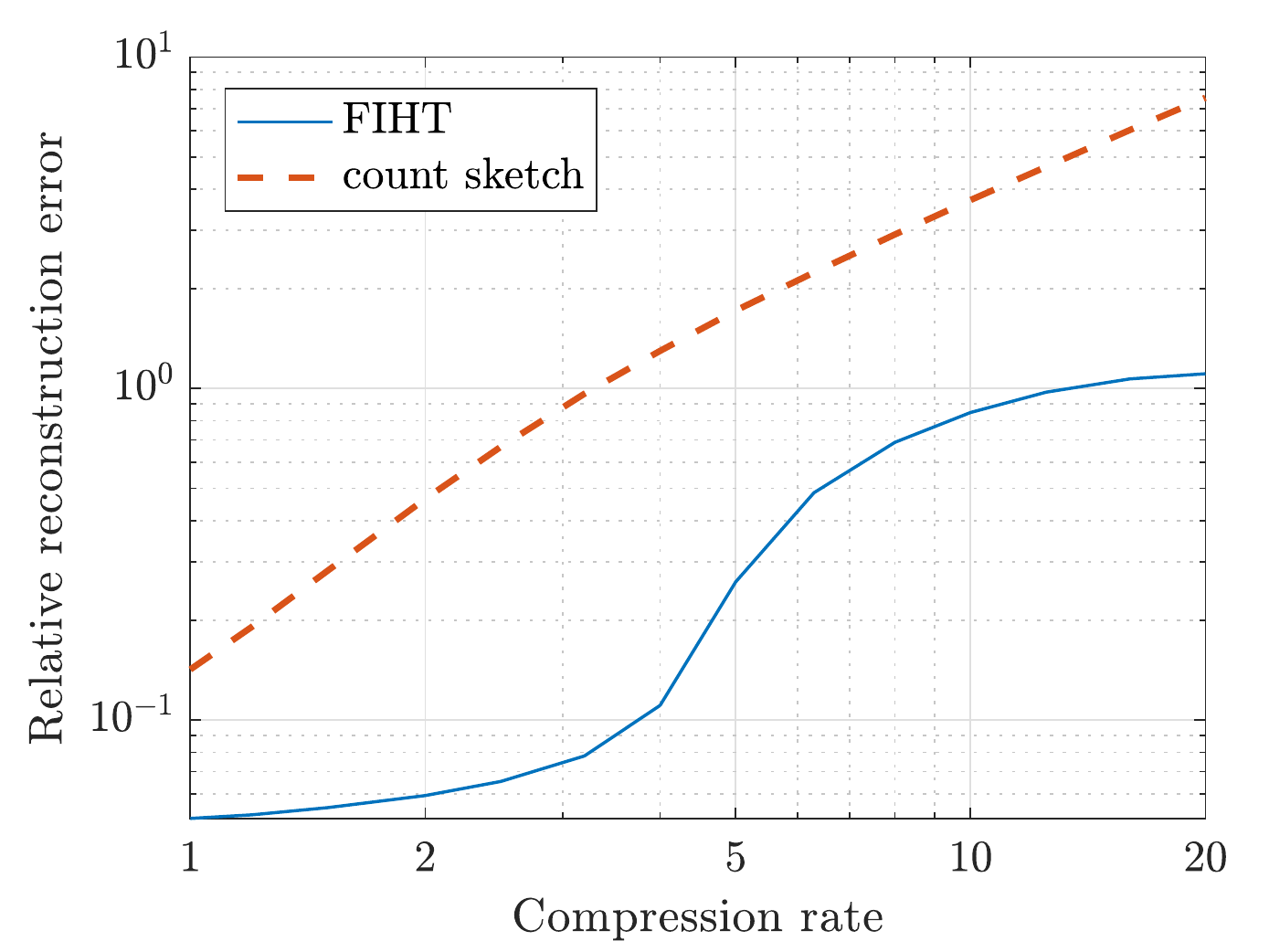}
\centering
\caption{Relative reconstruction error for FIHT and count sketch at different compression rates, averaged over $20$ random trials.}
\label{fig:recons}
\end{figure}

\section{For-each Scheme and For-all Scheme}\label{appendix:two_schemes}

In what follows, we elaborate on the two schemes in compressed sensing: the \emph{for-each} scheme and the \emph{for-all} scheme.
\paragraph{{For-each} scheme:} In this scheme, the sensing matrix $\Phi$ is chosen at random from some distribution for each individual signal $x$. Specifically, consider an algorithm belonging to the for-each scheme,  and let $\mathcal{A}(y;\Phi)$ denote the signal reconstructed from the measurement $y$ and the sensing matrix $\Phi$ by this algorithm. Then associated with the algorithm $\mathcal{A}$ is an indexed set $\{\mathcal{D}_{d,Q}\}$ with each $\mathcal{D}_{d,Q}$ being a probability distribution over $\mathbb{R}^{Q\times d}$, and theoretical guarantee of the algorithm $\mathcal{A}$ can be typically recast as follows:

\vspace{3pt}
{\it Let $K$ be sufficiently small [say, $K\leq O(Q/\log d)$]. Then there exist $\varsigma>0$ and $c\geq 0$ depending on $K$, $Q$ and $d$, such that
\begin{equation}\label{eq:app_foreach}
\mathbb{P}_{\Phi\sim\mathcal{D}_{d,Q}}
\!\left(
\left\|\mathcal{A}(\Phi x;\Phi)-x\right\|_2 \leq (1+\varsigma)\big\|x-x^{[K]}\big\|_2
\right)
\geq 1- O(d^{-c}),
\end{equation}
where $x\in\mathbb{R}^d$ is any arbitrary deterministic vector.
}
\vspace{3pt}

Note that the bound~\eqref{eq:app_foreach} applies to the reconstruction of one deterministic signal $x\in\mathbb{R}^d$, and the probability on the left-hand side of~\eqref{eq:app_foreach} is with respect to the distribution $\mathcal{D}_{d,Q}$ over the sensing matrix $\Phi$. Consequently, every time a new vector $x$ is to be measured and reconstructed under the for-each scheme, we typically generates a new sensing matrix $\Phi$ independent of $x$. Especially, if a randomly generated matrix $\Phi$ has already been used for the reconstruction of $x$, and $y$ is a vector that is dependent of $\mathcal{A}(\Phi x,\Phi)$, then theoretically the bound~\eqref{eq:app_foreach} can no longer be directly applied to the reconstruction of $y$ by $\mathcal{A}(\Phi y;\Phi)$, and as a result,  bounding $\big\|\mathcal{A}(\Phi y,\Phi)-y\big\|_2$ would be more challenging.

Examples of the for-each scheme include count sketch~\cite{charikar2002finding} and count-min sketch~\cite{cormode2005improved}.

\paragraph{{For-all} scheme:} In this scheme, one generates a single sensing matrix $\Phi\in\mathbb{R}^{Q\times d}$ that is used for the reconstruction of all possible signals $x\in\mathbb{R}^d$. As already mentioned, the \emph{restricted isometry property} (see Definition~\ref{def:RIP}) has been proposed as a condition on $\Phi$ under the for-all scheme. Intuitively, the restricted isometry property ensures that the linear measurements $y=\Phi x$ can discriminate sparse signals: Suppose $\Phi\in\mathbb{R}^{Q\times d}$ satisfies $(2K,\delta)$-RIP for some $K$ and $\delta\in(0,1)$. Then for any $x_1,x_2\in\mathbb{R}^d$ with $\|x_1\|_0\leq K$ and $\|x_2\|_0\leq K$, we have
$$
\|\Phi(x_1-x_2)\|_2
\geq \sqrt{1-\delta}\cdot \|x_1-x_2\|_2,
$$
indicating that the map $x\mapsto \Phi x$ is an injection from $\{x\in\mathbb{R}^d:\|x\|_0\leq K\}$ to $\mathbb{R}^d$. Consequently, given the measurement vector $y=\Phi x$, the solution to~\eqref{eq:comp_sens_form1} will always be unique and equal to $x$ as long as $\|x\|_0\leq K$. This explains why RIP can provide reconstruction guarantees under the for-all scheme in the ideal scenario (i.e., original vector $x$ is strictly sparse, measurement is noiseless, and the solution to the nonconvex problem~\eqref{eq:comp_sens_form1} can be obtained).

Extending the above discussion to more general scenarios is not trivial, but the existing literature has investigated reconstruction guarantees of various compressed sensing algorithms when the original vector $x$ is approximately sparse or linear measurement by $\Phi$ is noisy. One typical form of guarantees for compressed sensing algorithms under the for-all scheme is as follows:

\vspace{3pt}
{\it Suppose $\Phi\in\mathbb{R}^{Q\times d}$ satisfies $(cK,\delta)$-RIP for some $c>1$ and sufficiently small $\delta>0$. Then there exist $C_1>0$ and $C_2>0$ depending on $\delta$, such that
\begin{equation*}
\left\|\mathcal{A}(\Phi x+w;\Phi)-x\right\|_2 \leq \frac{C_1}{\sqrt{K}}\big\|x-x^{[K]}\big\|_1
+C_2\|w\|_2
\end{equation*}
for all $x\in\mathbb{R}^d$, where $\mathcal{A}(\Phi x+w;\Phi)$ denotes the reconstructed signal from the (noisy) measurement vector $\Phi x+w$ and the sensing matrix $\Phi$.
}
\vspace{3pt}

Examples of the for-all scheme include $\ell_1$-minimization~\cite{candes2005decoding} and various greedy algorithms~\cite{needell2009cosamp,wei2014fast,blanchard2015cgiht}; see also~\cite{foucart2012sparse}. We also mention that there are variants of RIP that have been used for compressed sensing with sparse sensing matrices~\cite{berinde2008combining,gilbert2010sparse}.

Finally, we note that in practice, sensing matrices that satisfy RIP are usually generated by randomized methods. For example, it has been shown that a $Q\times d$ matrix having i.i.d. standard Gaussian entries will satisfy $(K,\delta)$-RIP for $Q\geq O(K\log(d/K))$ with high probability~\cite{baraniuk2008simple}. In this work, we generate the sensing matrix $\Phi$ based on Proposition~\ref{prop:subsampled_Fourier} so that $\Phi$ has low storage and transmission cost.

\section{The Fast Iterative Hard Thresholding (FIHT) Algorithm}\label{appendix:FIHT}

We introduce some additional notations before presenting FIHT. For any $x\in\mathbb{R}^d$, we let $\mathsf{Supp}(x)$ denote the set of indices of nonzero entries of $x$, and let $\mathsf{PrincipalSupp}(x;K)$ denote the set of indices corresponding to the top-$K$ entries of $x$ in magnitude. Given $x\in\mathbb{R}^d$ and an index set $S\subseteq\{1,\ldots,d\}$, we use $\mathsf{Proj}(x,S)$ to denote the $d$-dimensional vector that keeps the entries of $x$ with indices in $S$ and sets other entries to be zero.

The Fast Iterative Hard Thresholding (FIHT) algorithm~\cite{wei2014fast} is outlined as follows:

\begin{algorithm}[H]
\renewcommand{\thealgocf}{}
\caption{Fast Iterative Hard Thresholding (FIHT)}
\DontPrintSemicolon
\KwIn{measurement $y$, sensing matrix $\Phi$, number of nonzero entries in the output $K$}

Initialize: $g(0)=0\in\mathbb{R}^d$, $w(0)=A^\top y$, $\Omega(0) = \mathsf{PrincipalSupp}(w(0),K)$, $g(1) = \mathsf{Proj}(w(0), \Omega(0))$, $s=1$\;

\Repeat{stopping criterion satisfied}{
\eIf{$s=1$}{
$\tau(s)=0$\;
}{
$\tau(s)=\mfrac{\langle y-\Phi g(s),\Phi(g(s)-g(s\!-\!1))\rangle}{\|\Phi(g(s)-g(s\!-\!1))\|^2}$\;
}
$w(s)= g(s)+\tau(s)(g(s)-g(s-1))$\;
$r^w(s) = \Phi^\top(y-\Phi w(s))$\;
$\Gamma(s) = \mathsf{Supp}(w(s))$\;
$\tilde{\alpha}(s) = \mfrac{\left\|\mathsf{Proj}(r^w(s), \Gamma(s))\right\|^2}{ \left\|\Phi\cdot \mathsf{Proj}(r^w(s), \Gamma(s))\right\|^2}$\;
$h(s) = w(s)+\tilde{\alpha}(s)r^w(s)$\;
$\Omega(s) = \mathsf{PrincipalSupp}(h(s),K)$\;
$\tilde{g}(s) = \mathsf{Proj}(h(s),\Omega(s))$\;
$r(s) = \Phi^\top(y-\Phi\tilde{g}(s))$\;
$\alpha(s) = \mfrac{\left\|\mathsf{Proj}(r(s),\Omega(s))\right\|^2}{\left\|\Phi\cdot\mathsf{Proj}(r(s),\Omega(s))\right\|^2}$\;
$g(s\!+\!1)=\tilde{g}(s)+\alpha(s)\cdot\mathsf{Proj}(r(s),\Omega(s))$\;
$s\leftarrow s+1$\;
}
\KwOut{$g(s)$}
\end{algorithm}

We elaborate further details on our implementation of FIHT.

\vspace{5pt}
\noindent\textbf{Stopping criterion.} In our implementation of FIHT, we stop the iterations whenever any of the following holds:
\begin{itemize}[leftmargin=12pt,itemsep=2pt,topsep=1pt]
\item $s\geq 26$ (i.e., the maximum number of iterations is $25$)
\item $\|w(s)\|_2\leq 10^{-4}$
\item $\operatorname{std}(\|w(s')\|_2:s\!-\!3\leq s'\leq s)\leq 0.01\cdot \operatorname{mean}(\|w(s')\|_2:s\!-\!3\leq s'\leq s)$

Here the tuple $(\|w(s')\|_2:s\!-\!3\leq s'\leq s)$ records the values of $\|w(s)\|_2$ in the last $4$ iterations, and $\operatorname{std}$ and $\operatorname{mean}$ denote the standard deviation and the mean value respectively.
\end{itemize}

\vspace{5pt}
\noindent\textbf{Multiplication of $\Phi$ with a vector.} We generate the sensing matrix $\Phi$ from an orthogonal matrix $B$ as in Proposition~\ref{prop:subsampled_Fourier}, and the ``base matrix'' $B$ is chosen to be either the DCT matrix
$$
B_{ij} = \sqrt{\frac{2}{d}}
\cdot
\frac{1}{
\sqrt{1+\delta_{1,i}}}
\cos\frac{\pi(i\!-\!1)(2j\!-\!1)}{2d},
\quad
1\leq i,j\leq d,
$$
or the WHT matrix $B=H^{(\log_2 d)}$ defined recursively by
$$
H^{(0)} = 1,
\qquad
H^{(k)}
=\frac{1}{\sqrt{2}}
\begin{bmatrix}
H^{(k-1)} & H^{(k-1)} \\
H^{(k-1)} & -H^{(k-1)}
\end{bmatrix},\quad k\geq 1,
$$
when $d$ is a power of $2$. Then, since both DCT and WHT have fast algorithms of time complexity $O(d\log d)$, the matrix-vector multiplication $\Phi u$ and $\Phi^\top u$ can be calculated with time complexity $O(d\log d)$ for any $u\in\mathbb{R}^d$. Specifically, suppose $\Sigma\subseteq\{1,\ldots,d\}$ is the index set of the $Q$ rows of $B$ that form $\Phi$. Taking DCT as an example, for any $u\in\mathbb{R}^d$, we can computed $\Phi u$ as follows:
\begin{enumerate}[leftmargin=15pt,topsep=1pt,itemsep=0pt]
\item $\widehat{u}\leftarrow \sqrt{d/Q}\cdot\mathsf{FastDCT}(u)$,
\item $\widehat{u}_j\leftarrow 0$ for each $j\in {\Sigma^c}$.
\end{enumerate}
Then $\widehat{u}$ is a vector whose $Q$ nonzero entries record the values of each entry of $\Phi u$. Similarly, for any $u\in\mathbb{R}^Q$, we can compute $\Phi^\top u$ by the following steps:
\begin{enumerate}[leftmargin=15pt, topsep=1pt,itemsep=0pt]
\item $\tilde{u}\leftarrow 0\in\mathbb{R}^d$,
\item $\tilde{u}_\Sigma\leftarrow u$, where $\tilde{u}_\Sigma$ denotes the subvector of $\tilde{u}$ with indices in $\Sigma$.
\item $v\leftarrow \sqrt{d/Q}\cdot \mathsf{FastInverseDCT}(\tilde{u})$
\end{enumerate}
Then we have $v=\Phi u$. Note that if we employ WHT instead of DCT, then both $\mathsf{FastDCT}$ and $\mathsf{FastInverseDCT}$ can be replaced by $\mathsf{FastWHT}$ since the WHT matrix is symmetric.

%%%%%%%%%%%%%%%%%%%%%%%%%%%%%%%%%%%%%%%%%%%%%%%%%%%%%%%%%%%%%%%%%%%%%%%%%%%%%%%%

\bibliographystyle{IEEEtran}
\bibliography{bibfile.bib}

\end{document}

%% file: introduction_arXiv.tex
Large-scale distributed stochastic optimization plays a fundamental role in the recent advances of machine learning, allowing models with vast sizes to be trained on massive datasets by multiple machines. In the meantime, the past few years have witnessed an explosive growth of networks of  IoT devices such as smart phones, self-driving cars, robots, unmanned aerial vehicles (UAVs), etc., which are capable of data collection and processing for many learning tasks. In many of these applications, due to privacy concerns, it is preferable that the local edge devices learn the model by cooperating with the central server but without sending their own data to the server. Moreover, the communication between the edge devices and the server is often through wireless channels, which are  lossy and unreliable in nature and have limited bandwidth, imposing significant challenges, especially for large dimensional problems.

To address the communication bottlenecks, researchers have investigated communication-efficient distributed optimization methods for large-scale problems, for both the device-server setting~\cite{mcmahan2017communication,magnusson2018communication} and the peer-to-peer setting~\cite{lee2018finite,reisizadeh2019exact}.
In this paper, we consider the device-server setting where a group of edge devices are coordinated by a central server.

Most existing techniques for the device-server setting can be classified into two categories. The first category aims to reduce the number of communication rounds, based on the idea that each edge device runs multiple local SGD steps in parallel before sending the local updates to the server for aggregation. This approach has also been called \texttt{FedAvg}~\cite{mcmahan2017communication} in federated learning and convergence has been studied in~\cite{stich2018local,wang2018cooperative,yu2019parallel}.
Another line of work investigates lazy/adaptive upload of information, i.e., local gradients are uploaded only when found to be informative enough~\cite{sun2020lazily}.

The second category focuses on efficient  compression of gradient information transmitted from edge devices to the server. Commonly adopted compression techniques include \emph{quantization}~\cite{alistarh2017qsgd,bernstein2018signsgd,magnusson2020maintaining} and \emph{sparsification}~\cite{stich2018sparsified,alistarh2018convergence,junshan2021federated}. These techniques can be further classified according to whether the gradient compression yields biased \cite{alistarh2017qsgd, junshan2021federated} or unbiased \cite{bernstein2018signsgd,alistarh2018convergence} gradient estimators. To handle the bias and boost convergence, \cite{stich2018sparsified,karimireddy2019error} introduced the error feedback method that accumulates and corrects the error caused by gradient compression at each step. 

Two recent papers~\cite{ivkin2019communication,rothchild2020fetchsgd} employ sketching methods for gradient compression. Specifically, each device compresses its local stochastic gradient by count sketch~\cite{charikar2002finding} via a common sketching operator; and the server recovers the indices and the values of large entries of the aggregated stochastic gradient from the gradient sketches. However, theoretical guarantees of count sketch were developed for recovering one \textit{fixed} signal by randomly generating a sketching operator from a given probability distribution. During SGD, gradient signals are constantly changing, making it impractical to generate a new sketch operator for every SGD iteration. Thus the papers apply a single sketching operator to all the gradients through the optimization procedure, while sacrificing theoretical guarantees. Further, there is a limited understanding of the performance when there is transmission error/noise of the uploading links. 

\vspace{8pt}
\noindent\textbf{Our Contributions.} We propose a distributed SGD-type algorithm that employs compressed sensing for gradient compression. Specifically, %at the heart of our proposed algorithm,
we adopt compressed sensing techniques for the compression of local stochastic gradients at the device side, and the reconstruction of the aggregated stochastic gradients at the server side.
The use of compressed sensing enables the server to approximately identify the top entries of the \emph{aggregated} gradient without querying directly each local gradient. Our algorithm also integrates error feedback strategies \emph{at the server side} to handle the bias introduced by compression, while keeping the edge devices to be \emph{stateless}. We provide convergence analysis of our algorithm in the presence of additive noise incurred by the uploading communication channels, and conduct numerical experiments that justify the effectiveness of our algorithm.

Besides the related work discussed above, it is worth noting that a recent paper~\cite{cai2020zeroth} uses compressed sensing for zeroth-order optimization, which exhibits a mathematical structure similar to this study. However, \cite{cai2020zeroth} considers the centralized setting and only establishes convergence to a neighborhood of the minimizer.

\vspace{8pt}
\noindent{\bf Notations:} For $x\in\mathbb{R}^d$, $\|x\|_p$ denotes its $\ell_p$-norm, and $x^{[K]}\in\mathbb{R}^d$ denotes its best-$K$ approximation, i.e., the vector that keeps the top $K$ entries of $x$ in magnitude with other entries set to~$0$.